\documentclass[reqno]{amsart}
\usepackage{amssymb}

\usepackage{amsthm}
\theoremstyle{plain}
\usepackage{tikz}
\usetikzlibrary{decorations.pathreplacing}
\tikzset{
mybrace/.style={decorate,decoration={brace,aspect=#1}}
}

\usepackage{enumerate,enumitem}
\usepackage{thmtools, thm-restate}
\declaretheorem{theorem}

\newcommand{\Z}{\mathbb{Z}}

\newcommand{\script}{\mathcal}
\newcommand{\parentheses}[1]{{\left( {#1} \right)}}

\newcommand{\p}{\parentheses}

\newcommand{\Set}[1]{{\left\lbrace {#1} \right\rbrace}}
\newcommand{\singleton}{\Set}

\newcommand{\pair}[1]{\langle {#1} \rangle}
\def\set#1:#2{\Set{{#1} \colon {#2}}}

\newcommand{\pseudogrid}[4]{\pair{{#1},{#2}}_{{#3},{#4}}}
\newcommand{\stdsquare}[3]{\blacksquare\p{{#1},{#2},{#3}} }
\newcommand{\doubleray}[2]{{\leftrightsquigarrow}\p{{#1},{#2}} }

\numberwithin{theorem}{section}
\newtheorem{lemma}[theorem]{Lemma}  
\newtheorem{cor}[theorem]{Corollary}

\newtheorem{prob}{Problem}
\newtheorem{conj}{Conjecture}

\newtheorem{claim}{Claim}

\theoremstyle{definition}
\newtheorem{defn}[theorem]{Definition}
\newtheorem{notation}[theorem]{Notation}

\newtheorem{step}{Step}

\newcommand{\N}{\mathbb{N}}

\begin{document}
\author{Joshua Erde}
\thanks{Joshua Erde was supported by the Alexander von Humboldt Foundation.}
\address{University of Hamburg, Department of Mathematics, Bundesstra{\ss}e 55 (Geomatikum), 20146 Hamburg, Germany}
\email{joshua.erde@uni-hamburg.de}
\author{Florian Lehner}
\thanks{Florian Lehner was supported by the Austrian Science Fund (FWF) Grant no. J 3850-N32}
\address{University of Warwick, Mathematics Institute, Zeeman Building, Coventry CV4 7AL, United Kingdom}
\email{mail@florian-lehner.net}
\author{Max Pitz}
\address{University of Hamburg, Department of Mathematics, Bundesstra{\ss}e 55 (Geomatikum), 20146 Hamburg, Germany}
\email{max.pitz@uni-hamburg.de}
\title{Hamilton decompositions of one-ended Cayley graphs}  



\begin{abstract}
We prove that any one-ended, locally finite Cayley graph with non-torsion generators admits a decomposition into edge-disjoint Hamiltonian (i.e.\ spanning) double-rays. In particular, the $n$-dimensional grid $\mathbb{Z}^n$ admits a decomposition into $n$ edge-disjoint Hamiltonian double-rays for all $n \in \N$.
\end{abstract}

\maketitle
\section{Introduction}

A \emph{Hamiltonian cycle} of a finite graph is a cycle which includes every vertex of the graph. A finite graph $G=(V,E)$ is said to have a \emph{Hamilton decomposition} if its edge set can be partitioned into disjoint sets $E=E_1 \dot\cup E_2 \dot\cup \cdots \dot\cup E_r$ such that each $E_i$ is a Hamiltonian cycle in $G$. 

The starting point for the theory of Hamilton decompositions is an old result by Walecki from 1890 according to which every finite complete graph of odd order has a Hamilton decomposition (see \cite{A08} for a description of his construction). Since then, this result has been extended in various different ways, and we refer the reader to the survey of Alspach, Bermond and Sotteau \cite{ABS90} for more information. 

Hamiltonicity problems have also been considered for infinite graphs, see for example the survey by Gallian and Witte \cite{WG84}. While it is sometimes not obvious which objects should be considered the correct generalisations of a Hamiltonian cycle in the setting of infinite graphs, for one-ended graphs the undisputed solution is to consider \emph{double-rays}, i.e.\ infinite, connected, 2-regular subgraphs. Thus, for us a \emph{Hamiltonian double-ray} is then a double-ray which includes every vertex of the graph, and we say that an infinite graph $G=(V,E)$ has a \emph{Hamilton decomposition} if we can partition its edge set into edge-disjoint Hamiltonian double-rays. 

In this paper we will consider infinite variants of two long-standing conjectures on the existence of Hamilton decompositions for finite graphs. The first conjecture concerns Cayley graphs: Given a finitely generated abelian group $(\Gamma,+)$ and a finite generating set $S$ of $\Gamma$, the \emph{Cayley graph} $G(\Gamma,S)$ is the multi-graph with vertex set $\Gamma$ and edge multi-set
\[
\{(x,x+g) \,: \,  x \in \Gamma, g \in S \}.
\]

\begin{conj}[Alspach \cite{A85}]\label{c:ab}
If $\Gamma$ is an abelian group and $S$ generates $G$, then the simplification of $G(\Gamma,S)$ has a Hamilton decomposition, provided that it is $2k$-regular for some $k$.
\end{conj}

Note that if $S \cap -S = \emptyset$, then $G(\Gamma,S)$ is automatically a $2|S|$-regular simple graph. If $G(\Gamma,S)$ is finite and $2$-regular, then the conjecture is trivially true. Bermond, Favaron and Maheo \cite{BFM89} showed that the conjecture holds in the case $k=2$. Liu \cite{L94} proved certain cases of the conjecture for finite $6$-regular Cayley graphs, and his result was further extended by Westlund \cite{W12}.

Our main theorem in this paper is the following affirmative result towards the corresponding infinite analogue of Conjecture~\ref{c:ab}:

\begin{theorem}\label{t:ZN}
Let $\Gamma$ be an infinite, finitely generated abelian group, and let $S$ be a generating set such that every element of $S$ has infinite order. If the Cayley graph $G=G(\Gamma,S)$ is one-ended, then it has a Hamilton decomposition.
\end{theorem}

We remark that under the assumption that elements of $S$ are non-torsion, the simplification of $G(\Gamma,S)$ is always isomorphic to a Cayley graph $G(\Gamma,S')$ with $S' \subseteq S$ and $S' \cap -S' = \emptyset$, and so our theorem implies the corresponding version of Conjecture~\ref{c:ab} for non-torsion generators, in particular for Cayley graphs of $\mathbb Z ^n$ with arbitrary generators.

In the case when $G=G(\Gamma,S)$ is two-ended, there are additional technical difficulties when trying to construct a decomposition into Hamiltonian double-rays. In particular, since each Hamiltonian double-ray must meet every edge cut an odd number of times, there can be parity reasons why no decomposition exists. One particular two-ended case, namely where $\Gamma \cong \mathbb{Z}$, has been considered by Bryant, Herke, Maenhaut and Webb \cite{BHMW17}, who showed that when $G(\mathbb{Z},S)$ is $4$-regular, then $G$ has a Hamilton decomposition unless there is an odd cut separating the two ends.

The second conjecture about Hamiltonicity that we consider concerns Cartesian products of graphs: Given two graphs $G$ and $H$ the \emph{Cartesian product} (or product) $G \square H$ is the graph with vertex set $V(G) \times V(H)$ in which two vertices $(g,h)$ and $(g',h')$ are adjacent if and only if either
\begin{itemize}
\item $g = g'$ and $h$ is adjacent to $h'$ in $H$, or
\item $h = h'$ and $g$ is adjacent to $g'$ in $G$.
\end{itemize}
Kotzig \cite{K73} showed that the Cartesian product of two cycles has a Hamilton decomposition, and conjectured that this should be true for the product of three cycles. Bermond extended this conjecture to the following:

\begin{conj}[Bermond \cite{B78}]\label{c:prod}
If $G_1$ and $G_2$ are finite graphs which both have Hamilton decompositions, then so does $G_1\square G_2$.
\end{conj}

Alspach and Godsil \cite{AG85} showed that the product of any finite number of cycles has a Hamilton decomposition, and Stong \cite{S91} proved certain cases of Conjecture~\ref{c:prod} under additional assumptions on the number of Hamilton cycles in the decomposition of $G_1$ and $G_2$ respectively.

Applying techniques we developed to prove Theorem~\ref{t:ZN}, we show as our second main result of this paper that Conjecture~\ref{c:prod} holds for countably infinite graphs.

\begin{restatable}{theorem}{main}\label{t:prod}
If $G$ and $H$ are countable graphs which both have Hamilton decompositions, then so does their product $G \square H$.
\end{restatable}

The paper is structured as follows: In Section~\ref{sec_group} we mention some group theoretic results and definitions we will need. In Section~\ref{sec_cov} we state our main lemma, the \emph{Covering Lemma}, and show that it implies Theorem~\ref{t:ZN}. The proof of the Covering Lemma will be the content of Section~\ref{sec_proof}. In Section \ref{sec_products} we apply our techniques to prove Theorem~\ref{t:prod}. Finally, in Section~\ref{sec_open} we list  open problems and possible directions for further work.

\section{Notation and preliminaries}\label{sec_group}

If $G=(V,E)$ is a graph, and $A,B \subseteq V$, we denote by $E(A,B)$ the set of edges between $A$ and $B$, i.e.\ $E(A,B) = \set{(x,y) \in E}:{x \in A, y \in B}$. For $A \subseteq V$ or $F \subseteq E$ we write $G[A]$ and $G[F]$ for the subgraph of $G$ induced by $A$ and $F$ respectively. 

For $A,B \subseteq \Gamma$ subsets of an abelian group $\Gamma$ we write $-A := \set{-a}:{a \in A}$ and $A + B := \set{a+b}:{a \in A, b \in B} \subseteq \Gamma$. If $\Delta$ is a subgroup of $\Gamma$, and $A \subset \Gamma$ a subset, then $A^\Delta = \set{a + \Delta}:{a \in A}$ denotes the family of corresponding cosets. If $g \in \Gamma$ we say that the \emph{order} of $g$ is the smallest $k \in \N$ such that $k\cdot g = 0$. If such a $k$ exists, then $g$ is a \emph{torsion element}. Otherwise, we say the order of $g$ is infinite and $g$ is a \emph{non-torsion} element. For $k \in \N$ we write $[k] = \Set{1,2,\ldots, k}$.

The following terminology will be used throughout.

\begin{defn}
\label{defn_isubgraph}
Given a graph $G$, an edge-colouring $c \colon E(G) \rightarrow [s]$ and a colour $i \in [s]$, the \emph{$i$-subgraph} is the subgraph of $G$ induced by the edge set $c^{-1}(i)$, and the \emph{$i$-components} are the components of the $i$-subgraph.
\end{defn}

\begin{defn}[Standard and almost-standard colourings of Cayley graphs]
\label{defn_colourings}
Let $\Gamma$ be an infinite abelian group, $S = \Set{g_1,g_2, \ldots , g_s}$ a finite generating set for $\Gamma$ such that every $g_i \in S$ has infinite order, and let $G$ be the Cayley graph $G(\Gamma,S)$.
\begin{itemize}
\item The \emph{standard colouring} of $G$ is the edge colouring $c_{\textnormal{std}}  \colon E(G) \rightarrow [s]$ such that $c_{\textnormal{std}}\big((x,x + g_i)\big) = i$ for each $x \in \Gamma, g_i \in S$. 
\item Given a subset $X \subseteq V(G)$ we say that a colouring $c$ is \emph{standard on $X$} if $c$ agrees with $c_{\textnormal{std}}$ on $G[X]$. Similarly if $F \subset E(G)$ we say that $c$ is \emph{standard on $F$} if $c$ agrees with $c_{\textnormal{std}}$ on $F$.
\item A colouring $c  \colon  E(G) \rightarrow [s]$ is \emph{almost-standard} if the following are satisfied:
\begin{itemize}
\item there is a finite subset $F \subseteq E(G)$ such that $c$ is standard on $E(G) \setminus F$;
\item for each $i \in [s]$ the $i$-subgraph is spanning, and each $i$-component is a double-ray.
\end{itemize}
\end{itemize}
\end{defn}

\begin{defn}[Standard squares and double-rays]
Let $\Gamma$ and $S$ be as above. Given $x\in \Gamma$ and $g_i \neq g_j \in S$, we call 
\[
\stdsquare{x}{g_i}{g_j} := \Set{(x,x+g_i),(x,x+g_j),(x+g_i ,x + g_i + g_j),(x+g_j,x + g_i + g_j)}
\]
an \emph{$(i,j)$-square with base point $x$}, and 
\[
\doubleray{x}{g_i} := \set{(x+n g_i,x+(n+1)g_i)}:{ n \in \Z}
\]
an \emph{$i$-double-ray with base point $x$}. 

Moreover, given a colouring $c \colon E(G(\Gamma,S)) \rightarrow [s]$ we call $\stdsquare{x}{g_i}{g_j}$ and $\doubleray{x}{g_i}$ an \emph{$(i,j)$-standard square} and \emph{$i$-standard double-ray} if $c$ is standard on 
$\stdsquare{x}{g_i}{g_j}$ and $\doubleray{x}{g_i}$ respectively.
\end{defn}

Since $\Gamma$ is an abelian group, every $\stdsquare{x}{g_i}{g_j}$ is a $4$-cycle in $G(\Gamma,S)$ (provided $g_i \neq -g_j$), and since $S$ contains no torsion elements of $\Gamma$, $\doubleray{x}{g_k}$ really is a double-ray in the Cayley graph $G(\Gamma,S)$.

Let $\Gamma$ be a finitely generated abelian group. By the Classification Theorem for finitely generated abelian groups (see e.g.\ \cite{F15}), there are integers $n,q_1,\ldots,q_r$ such that $\Gamma \cong \mathbb{Z}^n \oplus \bigoplus_{i=1}^r \mathbb{Z}_{q_i}$, where $\mathbb{Z}_{q}$ is the additive group of the integers modulo $q$. In particular, for each $\Gamma$ there is an integer $n$ and a finite abelian group $\Gamma_{\text{fin}}$ such that $\Gamma \cong \mathbb{Z}^n \oplus \Gamma_{\text{fin}}$.

The following structural theorem for the ends of finitely generated abelian groups is well-known:

\begin{theorem}\label{thm:one}
For a finitely generated group $\Gamma \cong \mathbb{Z}^n \oplus \Gamma_{\text{fin}}$, the following are equivalent:
\begin{itemize}
\item $n \geq 2$,
\item there exists a finite generating set $S$ such that $G(\Gamma,S)$ is one-ended, and
\item for all finite generating sets $S$, the Cayley graph $G(\Gamma,S)$ is one-ended.
\end{itemize}
\end{theorem}

\begin{proof}
See e.g.\ \cite[Proposition~5.2]{scott1979topological} for the fact the number of ends of $G(\Gamma,S)$ is independent of the choice of the generating set $S$, and \cite[Theorem~5.12]{scott1979topological} for the equivalence with the first item.
\end{proof}

A group $\Gamma$ satisfying one of the conditions from Theorem~\ref{thm:one} is called \emph{one-ended}.

\begin{cor}\label{c:oneend}
Let $\Gamma$ be an abelian group, $S=\{g_1,\ldots, g_s\}$ be a finite generating set such that the Cayley graph $G(\Gamma,S)$ is one-ended. Then, for every $g_i \in S$ of infinite order, there is some $g_j \in S$ such that $\pair{g_i,g_j} \cong (\Z^2,+)$.
\end{cor}
\begin{proof}
Suppose not. It follows that in $\Gamma / \langle g_i \rangle$ every element has finite order, and since it is also finitely generated, it is some finite group $\Gamma_f$ such that $\Gamma \cong \Z \oplus \Gamma_f$. Thus, by Theorem~\ref{thm:one}, $G$ is not one-ended, a contradiction.
\end{proof}

\section{The covering lemma and a high-level proof of Theorem~\ref{t:ZN}}\label{sec_cov}

Every Cayley graph $G(\Gamma,S)$ comes with a natural edge colouring $c_{\textnormal{std}}$, where we colour an edge $(x,x+g_i)$ with $x \in \Gamma$ and $g_i \in S$ according to the index $i$ of the corresponding generating element $g_i$. If every element of $S$ has infinite order, then every $i$-subgraph of $G(\Gamma,S)$ consists of a spanning collection of edge-disjoint double-rays, see  Definitions~\ref{defn_isubgraph} and \ref{defn_colourings}. So, it is perhaps a natural strategy to try to build a Hamiltonian decomposition by combining each of these monochromatic collections of double-rays into a single monochromatic spanning double-ray. 

Rather than trying to do this directly, we shall do it in a series of steps: given any colour $i \in [s] = |S|$ and any finite set $X \subset V(G)$, we will show that one can change the standard colouring at finitely many edges so that there is one particular double-ray in the colour $i$ which covers $X$. Moreover, we can ensure that the resulting colouring maintains enough of the structure of the standard colouring that we can repeat this process inductively: it should remain \emph{almost standard}, i.e.\ all monochromatic components are still double-rays, see Definition~\ref{defn_colourings}. By taking a sequence of sets $X_1 \subseteq X_2 \subseteq \cdots$ exhausting the vertex set of $G$, and varying which colour $i$ we consider, we ensure that in the limit, each colour class consists of a single spanning double-ray, giving us the desired Hamilton decomposition.

In this section, we formulate our key lemma, namely the Covering Lemma~\ref{lem_mainlemma}, which allows us to do each of these steps. We will then show how Theorem~\ref{t:ZN} follows from the Covering Lemma. The proof of the Covering Lemma is given in Section \ref{sec_proof}.

\begin{lemma}[Covering lemma]
\label{lem_mainlemma}
Let $\Gamma$ be an infinite, one-ended abelian group, $S = \Set{g_1,g_2, \ldots , g_s}$ a finite generating set such that every $g_i \in S$ has infinite order, and $G=G(\Gamma,S)$ the corresponding Cayley graph.

Then for every almost-standard colouring $c$ of $G$, every colour $i$ and every finite subset $X \subseteq V(G)$, there exists an almost-standard colouring $\hat{c}$ of $G$ such that
\begin{itemize}
\item $\hat{c}=c$ on $E(G[X])$, and
\item some $i$-component in $\hat{c}$ covers $X$.
\end{itemize}
\end{lemma}

\begin{proof}[Proof of Theorem~\ref{t:ZN} given Lemma~\ref{lem_mainlemma}]
Fix an enumeration $V(G)=\set{v_n}:{n \in \N}$. Let $X_0 = D'_0 = \{v_0\}$ and $c_0 = c_{\textnormal{std}}$. For each $n \geq 1$ we will recursively construct almost standard colourings $c_n \colon E(G) \to [s]$, finite subsets $X_n \subset V(G)$, ($n$ mod $s$)-components $D_n$ of $c_n$ and finite paths $D'_n \subseteq D_n$ such that for every $n \in \N$
\begin{enumerate}
\item $X_{n-1} \cup \singleton{v_n} \subseteq X_{n}$,
\item $V(D'_{n-1}) \subseteq X_n$,
\item $X_n \subseteq V(D'_n)$,
\item $D'_n$ properly extends the path $D'_{n - s}$ (the `previous' path of colour $n$ mod $s$) in both endpoints of $D'_{n - s}$, and
\item $c_{n}$ agrees with $c_{n-1}$ on $E(G[X_{n}])$.
\end{enumerate}

Suppose inductively for some $n \in \N$ that $c_n$, $X_n$, $D_n$ and $D'_n$ have already been defined. Choose some $X_{n +1} \supseteq X_n \cup \singleton{v_n}$ large enough such that (1) and (2) are satisfied. Applying Lemma~\ref{lem_mainlemma} with input $c_n$ and $X_{n+1}$ provides us with a colouring $c_{n+1}$ such that (5) is satisfied and some ($n+1$ mod $s$)-component $D_{n+1}$ covers $X_{n+1}$. Since $c_{n+1}$ is almost standard, $D_{n+1}$ is a double-ray. Furthermore, since $c_{n+1}$ agrees with $c_n$ on $E(G[X_{n+1}])$, by the inductive hypothesis it agrees with $c_k$ on $E(G[X_{k+1}])$ for each $k \leq n$.

Therefore, since $D'_{n+1-s} \subset X_{n-s+2}$ is a path of colour ($n+1$ mod $s$) in $c_{n+1-s}$, it follows that $D'_{n+1-s} \subset D_{n+1}$ and so we can extend $D'_{n+1-s}$ to a sufficiently long finite path $D'_{n+1} \subset D_{n+1}$ such that (3) and (4) are satisfied at stage $n+1$.

Once the construction is complete, we define $T_1, \ldots, T_{s} \subset G$ by 
$$T_i = \bigcup_{n \equiv i \mod s } D'_n$$
and claim that they form a decomposition of $G$ into edge-disjoint Hamiltonian double-rays. Indeed, by (4), each $T_i$ is a double-ray. That they are edge-disjoint can be seen as follows: Suppose for a contradiction that $e \in E(T_i) \cap E(T_j)$. Choose $n(i)$ and $n(j)$ minimal such that $e \in E(D'_{n(i)}) \subset E(T_i)$ and $e \in E(D'_{n(j)}) \subset E(T_j)$. We may assume that $n(i) < n(j)$, and so $e \in E(G[X_{n(i)+1}])$ by (2). Furthermore, by $(5)$ it follows that $c_{n(j)}$ agrees with $c_{n(i)}$ on $E(G[X_{n(i)+1}])$. However by construction $c_{n(j)}(e) = j \neq i = c_{n(i)}(e)$ contradicting the previous line.

Finally, to see that each $T_i$ is spanning, consider some $v_n \in V(G)$. By (1), $v_n \in X_{n}$. Pick $n' \geq n$ with $n' \equiv i \mod s$. Then by (3), $D'_{n'} \subset T_i$ covers $X_{n'}$ which in turn contains $v_n$, as $v_n \in X_{n} \subseteq X_{n'}$ by (1).
\end{proof}

\section{Proof of the Covering Lemma}\label{sec_proof}

\subsection{Blanket assumption.} Throughout this section, let us now fix 
\begin{itemize}
\item a one-ended infinite abelian group $\Gamma$ with finite generating set $S = \Set{g_1, \ldots, g_s}$ such that every element of $S$ has infinite order,
\item an almost-standard colouring $c$ of the Cayley graph $G=G(\Gamma,S)$,
\item a finite subset $X \subseteq \Gamma$  such that $c$ is standard on $V(G) \setminus X$,
\item a colour $i$, say $i = 1$, and corresponding generator $g_1 \in S$, for which we want to show Lemma \ref{lem_mainlemma}, and finally
\item a second generator in $S$, say $g_2$, such that $\Delta:=\pair{g_1,g_2} \cong (\Z^2,+)$, see Corollary~\ref{c:oneend}. 
\end{itemize}

\subsection{Overview of proof}
We want to show Lemma \ref{lem_mainlemma} for the Cayley graph $G$, colouring $c$, generator $g_1$ and finite set $X$. The cosets of $\pair{g_1,g_2}$ in $\Gamma$ cover $V(G)$, and in the standard colouring the edges of colour $1$ and $2$ form a grid on $\pair{g_1,g_2}$. So, since $c$ is almost-standard, on each of these cosets the edges of colour $1$ and $2$ will look like a grid, apart from on some finite set.

Our aim is to use the structure in these grids to change the colouring $c$ to one satisfying the conclusions of Lemma \ref{lem_mainlemma}. It will be more convenient to work with large finite grids, which we require, for technical reasons, to have an even number of rows. This is the reason for the slight asymmetry in the definition below.

\begin{notation}
Let $g_i,g_j \in \Gamma$. For $N,M \in \N$ we write
$$\pseudogrid{g_i}{g_j}{N}{M} := \set{n g_i + m g_j}:{n,m \in \Z, \; -N \leq n \leq N, \; -M < m \leq M} \subseteq \pair{g_i,g_j} \subseteq \Gamma.$$
\end{notation}

The structure of our proof can be summarised as follows. First, in Section~\ref{subsec_1}, we will show that there is some $N_0$ and some `nice' finite set of $P$ of representatives of cosets of $\pair{g_1,g_2}$ such that $P + \pseudogrid{g_1}{g_2}{N_0}{N_0}$ covers $X$. We will then, in Section~\ref{subsec_1.5} pick sufficiently large numbers $N_0 < N_1 < N_2 < N_3$ and consider the grids $P + \pseudogrid{g_1}{g_2}{N_3}{N_1}$. Using the structure of the grids we will make local changes to the colouring inside $P + (\pseudogrid{g_1}{g_2}{N_3}{N_1} \setminus \pseudogrid{g_1}{g_2}{N_0}{N_0})$ to construct our new colouring $\hat{c}$. This new colouring $\hat{c}$ will then agree with $c$ on the subgraph induced by $P + \pseudogrid{g_1}{g_2}{N_0}{N_0} \supseteq X$, and be standard on $V(G) \setminus \big(P + \pseudogrid{g_1}{g_2}{N_3}{N_1} \big)$, and hence, as long as we ensure all the colour components are double-rays, almost-standard.

These local changes will happen in three steps. First, in Step~\ref{step_co}, we will make local changes inside $x_\ell + (\pseudogrid{g_1}{g_2}{N_3}{N_1} \setminus \pseudogrid{g_1}{g_2}{N_2}{N_1})$ for each $x_\ell \in P$, in order to make every $i$-component meeting $P + \pseudogrid{g_1}{g_2}{N_2}{N_1}$ a finite cycle. 

Next, in Step~\ref{step_cc}, we will make local changes inside $x_\ell + (\pseudogrid{g_1}{g_2}{N_2}{N_1} \setminus \pseudogrid{g_1}{g_2}{N_1}{N_1})$ for each $x_\ell \in P$, in order to combine the cycles meeting this translate of the grid into a single cycle. 

Finally, in Step~\ref{step_mainlemma_quant}, we will make local changes inside $P + (\pseudogrid{g_1}{g_2}{N_1}{N_1} \setminus \pseudogrid{g_1}{g_2}{N_0}{N_0})$, in order to join the cycles for different $x_\ell$ into a single cycle covering $P + \pseudogrid{g_1}{g_2}{N_0}{N_0}$. We then make one final local change to turn this finite cycle into a double-ray.

\subsection{Identifying the relevant cosets}
\label{subsec_1}

\begin{lemma}
\label{lem_coveringpath}
There exist $N_0 \in \N$ and a finite set $P = \Set{x_0, \ldots, x_t} \subset \Gamma$ such that
\begin{itemize}
\item $P^\Delta = \Set{x_0 + \Delta, \ldots, x_t + \Delta}$ is a path in $G(\Gamma / \Delta, \p{S \setminus \Set{g_1,g_2}}^\Delta)$, and
\item $X \subseteq P + \pseudogrid{g_1}{g_2}{N_0}{N_0}$.
\end{itemize}
\end{lemma}
\begin{proof}
Since $X$ is finite, there is a finite set $Y = \Set{y_1, \ldots, y_k} \subset \Gamma$ such that the cosets in $Y^\Delta=\Set{y_1 + \Delta, \ldots, y_k + \Delta}$ are all distinct and cover $X$. Moreover, since every $(y_\ell + \Delta ) \cap X$ is finite, there exists $N_0 \in \N$ such that 
$$(y_\ell + \pair{g_1,g_2}) \cap X = (y_\ell + \pseudogrid{g_1}{g_2}{N_0}{N_0} ) \cap X$$
for all $1 \leq \ell \leq k$. Then $X \subseteq Y + \pseudogrid{g_1}{g_2}{N_0}{N_0}$.

Next, by a result of Nash-Williams \cite{N59}, every Cayley graph of a countably infinite abelian group has a Hamilton double-ray, and it is a folklore result (see \cite{WG84}) that every Cayley graph of a finite abelian group has a Hamilton cycle. So in particular, the Cayley graph of $(\Gamma / \Delta, \p{S \setminus \Set{g_1,g_2}}^\Delta)$, has a Hamilton cycle / double-ray, say $H$. Let $P \supseteq Y$ be a finite set of representatives of the cosets of $\Delta$ which lie on the convex hull of $Y^\Delta$ on $H$. It is clear that $P$ is as required.
\end{proof}

\begin{itemize}
\item For the rest of this section let us fix $N_0 \in \mathbb{N}$ and $P = \Set{x_0, \ldots, x_t} \subset \Gamma$ to be as given by Lemma \ref{lem_coveringpath}.
\end{itemize}

\subsection{Picking sufficiently large grids}
\label{subsec_1.5}

In order to choose our grids large enough to be able to make all the necessary changes to our colouring, we will first need the following lemma, which guarantees that we can find, for each $k \neq 1,2$ and $x \in \Gamma$, many distinct standard $k$-double-rays which go between the cosets $x + \Delta$ and $(x + g_k) + \Delta$.

\begin{lemma}
\label{l_enoughuntamperedrays}
For any $g_k \in S \setminus \singleton{g_1,g_2}$ and any pair of distinct cosets $x+\Delta$ and $(g_k+x)+\Delta$, there are infinitely many distinct standard $k$-double-rays $R$ for the colouring $c$ with $E(R) \cap E(x+\Delta,(g_k+x)+\Delta) \neq \emptyset$.
\end{lemma}

\begin{proof}
It clearly suffices to prove the assertion for $c = c_\textnormal{std}$. We claim that either
\[\script{R}_1=\set{\doubleray{x+mg_1}{g_k}}:{m \in \Z} \; \text{ or } \; \script{R}_2=\set{\doubleray{x+mg_2}{g_k}}:{m \in \Z} \]
is such a collection of disjoint standard $k$-double-rays.

Suppose that $\script{R}_1$ is not a collection of disjoint double-rays. Then there are $m \neq m' \in \Z$ and $n,n' \in \Z$ such that 
$$m g_1 +  n  g_k = m' g_1 +  n'  g_k.$$
Since $g_k$ has infinite order, it follows that $n \neq n'$, too, and so we can conclude that there are $\ell, \ell' \in \Z \setminus \singleton{0}$ such that $\ell g_1 = \ell' g_k.$ Similarly, if $\script{R}_2$ was not a collection of disjoint double-rays, then we can find $q, q' \in \Z \setminus \singleton{0}$ such that $q g_2 = q' g_k.$ However, it now follows that 
$$q'\ell g_1 = q'(\ell' g_k) = \ell' (q'g_k) = \ell' q g_2,$$
contradicting the fact that $\langle g_1, g_2 \rangle \cong (\Z^2,+)$. This establishes the claim.

Finally, observe that if say $\script{R}_1$ is a disjoint collection, then for every $R_m = \doubleray{x+mg_1}{g_k} \in \script{R}_1$ we have $(x + m g_1,x + m g_1 + g_k) \in E(R_m) \cap E(x+\Delta,(g_k+x)+\Delta)$ as desired.
\end{proof}

We are now ready to define our numbers $N_0 < N_1 < N_2 < N_3$. Recall that $N_0$ and $P=\Set{x_0, \ldots, x_t}$ are given by Lemma \ref{lem_coveringpath}. For each $\ell \in [t]$, let $g_{n(\ell)}$ be some generator in $S \setminus \Set{g_1,g_2}$ that induces the edge between $x_{\ell-1} + \Delta$ and $x_\ell+ \Delta$ on the path $P^\Delta$. Note that $n(\ell) \in [s] \setminus \Set{1,2}$ for all $\ell$. 

By Lemma~\ref{l_enoughuntamperedrays}, we may find $t^2$ many disjoint standard double-rays
\[ \script{R} = \set{R^k_\ell}:{1 \leq k,\ell \leq t} \]
such that for every $\ell$, the double-rays in $\set{R^k_\ell = \doubleray{y^k_\ell}{g_{n(\ell)}}}:{k \in [t]}$ are standard $n(\ell)$-double-rays containing an edge
$$e^k_\ell=(y^k_\ell,y^k_\ell + g_{n(\ell)}) \in E(R^k_\ell) \cap E(x_{\ell-1}+\Delta, x_\ell + \Delta)$$
so that all
$T^k_\ell = \stdsquare{y^k_\ell}{g_i}{g_{n(\ell)}}$
are $(1,n(\ell))$-standard squares for $c$ which have empty intersection with $\Set{x_{\ell-1},x_\ell} + \pseudogrid{g_1}{g_2}{N_0}{N_0}$. Furthermore we may assume that these standard squares are all edge-disjoint. Then
\begin{itemize}
\item let $N_1 > N_0$ be sufficiently large such that the subgraph induced by $P + \pseudogrid{g_1}{g_2}{N_1-3}{N_1-3}$ contains all standard squares $T^k_\ell$ mentioned above. 
\item Let $N_2$ be arbitrary with $N_2 \geq 5N_1$.
\item Let $N_3$ be arbitrary with $N_3 \geq N_2 + 2N_1$.
\end{itemize}

\subsection{The cap-off step}
\label{subsec_2}

Our main tool for locally modifying our colouring is the following notion of `colour switchings', which is also used in \cite{L94}.

\begin{defn}[Colour switching of standard squares]
Given an edge colouring $c \colon E(G(\Gamma,S)) \to [s]$ and an $(i,j)$-standard square $\stdsquare{x}{g_i}{g_j}$,
a \emph{colour switching} on $\stdsquare{x}{g_i}{g_j}$ changes the colouring $c$ to the colouring $c'$ such that
\begin{itemize}
\item $c'=c$ on $E \setminus \stdsquare{x}{g_i}{g_j}$,
\item $c'\big((x,x+g_i)\big) = c'\big((x + g_j,x+ g_i + g_j)\big) = j$,
\item $c'\big((x,x+g_j)\big) = c'\big((x + g_i,x+ g_i + g_j)\big) = i$.
\end{itemize}
\end{defn}

It would be convenient if colour switchings maintained the property that a colouring is almost-standard. Indeed, if $c$ is standard on $E(G) \setminus F$ then $c'$ is standard on $E(G) \setminus (F \cup \stdsquare{x}{g_i}{g_j})$. Also, it is a simple check that if the $i$ and $j$-subgraphs of $G$ for $c$ are $2$-regular and spanning, then the same is true for $c'$. However, some $i$ or $j$-components may change from double-rays to finite cycles, and vice versa. 

\begin{step}[Cap-off step]
\label{step_co}
There is a colouring $c'$ obtained from $c$ by colour switchings of finitely many $(1,2)$-standard squares such that
\begin{itemize}
\item $c'=c$ on $E(G[X])$;
\item every $1$-component in $c'$ meeting $P + \pseudogrid{g_1}{g_2}{N_2}{N_1}$ is a finite cycle intersecting both $P + (\pseudogrid{g_1}{g_2}{N_3}{N_1} \setminus \pseudogrid{g_1}{g_2}{N_2}{N_1})$ and $P + \pseudogrid{g_1}{g_2}{N_1}{N_1}$;
\item every other $1$-component, and all other components of all other colour classes of $c'$ are double-rays;
\item $c'$ is standard outside of $P + \pseudogrid{g_1}{g_2}{N_3}{N_1}$ and inside of $P + (\pseudogrid{g_1}{g_2}{N_2}{N_1} \setminus \pseudogrid{g_1}{g_2}{N_0}{N_0})$;
\item  for each $x_\ell \in P$, the sets of vertices
\[
\Set{{x_l + n g_1 + mg_2}\colon{  N_1 \leq |n| \leq N_2 , m \in \{ N_1, N_1 -1\}}}
\]
are each contained in a single $1$-component of $c'$.
\end{itemize} 
\end{step}

\begin{proof}
For $\ell \in [t]$ and $q \in [N_1]$ let 
$R^\ell_q =\stdsquare{v^\ell_q}{g_1}{g_2}$ and $L^\ell_q =\stdsquare{w^\ell_q}{g_1}{g_2}$
be the $(1,2)$-squares with base point $v^\ell_q = x_\ell + (N_3 + 1 -2q) \cdot g_1 + (N_1 + 1 - 2q) \cdot g_2$ and $w^\ell_q = x_\ell - (N_3 + 2 -2q) \cdot g_1 + (N_1 + 1 - 2q) \cdot g_2$ respectively. The square $L^\ell_q$ is the mirror image of $R^\ell_q$ with respect to the $y$-axis of the grid $x_\ell + \pair{g_1,g_2}$, however the base points are not mirror images, accounting for the slight asymmetry in the definitions.

Since $N_3 \geq N_2 + 2N_1$, it follows that 
\[R^\ell_q \cup L^\ell_q \subseteq E(x_\ell + (\pseudogrid{g_1}{g_2}{N_3}{N_1} \setminus \pseudogrid{g_1}{g_2}{N_2}{N_1}))\]
for all $q \in [N_1]$, and so by assumption on $c$, all $R^\ell_q$ and $L^\ell_q$ are indeed standard $(1,2)$-squares.
We perform colour switchings on $R^\ell_q$ and $L^\ell_q$ for all $\ell \in [t]$ and $q \in [N_1]$, and call the resulting edge colouring $c'$. It is clear that $c'=c$ on $E(G[X])$ and that $c'$ is standard outside of $P + \pseudogrid{g_1}{g_2}{N_3}{N_1}$ and inside of $P + (\pseudogrid{g_1}{g_2}{N_2}{N_1} \setminus \pseudogrid{g_1}{g_2}{N_0}{N_0})$.
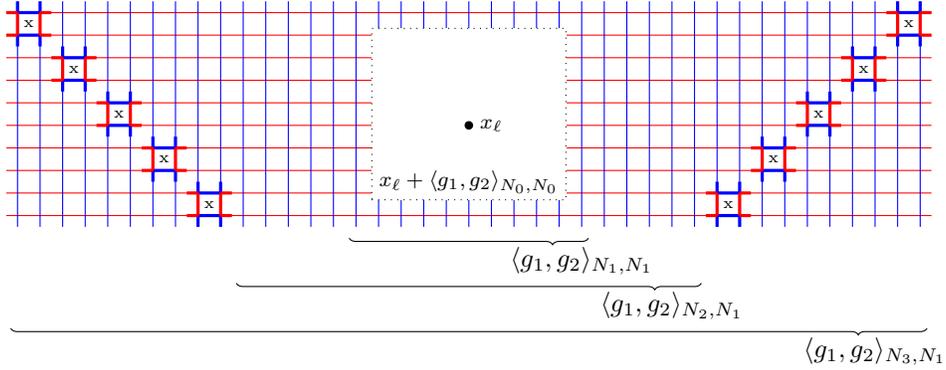
\begin{figure}[ht!]
\centering
\begin{tikzpicture}[scale=.3]

\def \NZero {4}
\def \NOne {5}	
\def \NThree {20} 

\pgfmathsetmacro\rows{2*\NOne} 
\foreach \n in {1,...,\rows}
{
\pgfmathsetmacro\y{\n-\NOne}
\draw[red, thin] (-\NThree-.5,\y) -- (\NThree+.5,\y);
}

\foreach \n in {-\NThree,...,\NThree}
{
\draw[blue, thin] (\n,-\NOne+.5) -- (\n,\NOne+.5);
}

\draw[mybrace=0.15] (\NOne+.3,-\NOne) -- (-\NOne-.3,-\NOne);
\node at (\NOne,-\NOne-1) {$\pseudogrid{g_1}{g_2}{N_1}{N_1}$};

\draw[mybrace=0.1] (\NThree-\rows+.3,-\NOne-2) -- (-\NThree+\rows-.3,-\NOne-2);
\node at (\NThree-\rows-1,-\NOne-3) {$\pseudogrid{g_1}{g_2}{N_2}{N_1}$};

\draw[mybrace=0.075] (\NThree+.3,-\NOne-4) -- (-\NThree-.3,-\NOne-4);
\node at (\NThree-2,-\NOne-5) {$\pseudogrid{g_1}{g_2}{N_3}{N_1}$};

\draw [fill=white, dotted] (-\NZero-.3,-\NZero+.7) rectangle (\NZero+.3,\NZero+.3);
\node at (0,-\NZero+1.5) {\footnotesize{$x_\ell + \pseudogrid{g_1}{g_2}{N_0}{N_0}$}};
\node[draw, circle,scale=.3, fill] (xl) at (0,0) {};
\node at (1,0) {\footnotesize{$x_\ell$}};

\foreach \q in {1,...,\NOne}
{
\pgfmathsetmacro\x{\NThree + 1 - 2*\q}
\pgfmathsetmacro\y{\NOne + 1 - 2*\q}
\node at (\x+.5,\y+.5) {$\textnormal{\tiny{x}}$};
\draw [blue,very thick] (\x,\y) -- (\x+1,\y);
\draw [blue,very thick] (\x,\y) -- (\x,\y-.5);
\draw [blue,very thick] (\x+1,\y) -- (\x+1,\y-.5);

\draw [blue,very thick] (\x,\y+1) -- (\x+1,\y+1);
\draw [blue,very thick] (\x,\y+1) -- (\x,\y+1.5);
\draw [blue,very thick] (\x+1,\y+1) -- (\x+1,\y+1.5);

\draw [red,very thick] (\x,\y) -- (\x,\y+1);
\draw [red,very thick] (\x,\y) -- (\x-.5,\y);
\draw [red,very thick] (\x,\y+1) -- (\x-.5,\y+1);

\draw [red,very thick] (\x+1,\y) -- (\x+1,\y+1);
\draw [red,very thick] (\x+1,\y) -- (\x+1.5,\y);
\draw [red,very thick] (\x+1,\y+1) -- (\x+1.5,\y+1);
}
\foreach \q in {1,...,\NOne}
{
\pgfmathsetmacro\x{-(\NThree + 1 - 2*\q)-1}
\pgfmathsetmacro\y{\NOne + 1 - 2*\q}
\node at (\x+.5,\y+.5) {$\textnormal{\tiny{x}}$};
\draw [blue,very thick] (\x,\y) -- (\x+1,\y);
\draw [blue,very thick] (\x,\y) -- (\x,\y-.5);
\draw [blue,very thick] (\x+1,\y) -- (\x+1,\y-.5);

\draw [blue,very thick] (\x,\y+1) -- (\x+1,\y+1);
\draw [blue,very thick] (\x,\y+1) -- (\x,\y+1.5);
\draw [blue,very thick] (\x+1,\y+1) -- (\x+1,\y+1.5);

\draw [red,very thick] (\x,\y) -- (\x,\y+1);
\draw [red,very thick] (\x,\y) -- (\x-.5,\y);
\draw [red,very thick] (\x,\y+1) -- (\x-.5,\y+1);

\draw [red,very thick] (\x+1,\y) -- (\x+1,\y+1);
\draw [red,very thick] (\x+1,\y) -- (\x+1.5,\y);
\draw [red,very thick] (\x+1,\y+1) -- (\x+1.5,\y+1);
}

\end{tikzpicture}
\caption{Performing colour switchings of standard squares at positions indicated by `x' in a  copy $x_\ell + \pseudogrid{g_1}{g_2}{N_3}{N_1}$ of a finite grid.}
\label{fig_capoff}
\end{figure}
Let $C \subset G$ denote the region consisting of all vertices that lie in $x_\ell + (\pseudogrid{g_1}{g_2}{N_3}{N_1}$ for some $\ell$ between a pair $L^\ell_q$ and $R^\ell_q$ for some $q$, i.e.
\[
C = \bigcup_{\ell=1}^{t}\bigcup_{q=1}^{N_1} \bigcup_{m=1}^2 \set{x_\ell + n g_1 +  (N_1 + m- 2q)  g_2}:{|n| \leq N_3 + 1 -2q}.
\]
Then $P + \pseudogrid{g_1}{g_2}{N_2}{N_1} \subseteq C$. By construction, there are no edges of colour $1$ in $c'$ leaving $C$, that is, $E(C,V(G) \setminus C) \cap c'^{-1}(1) = \emptyset$. In particular, since the $1$-subgraph of $G$ under $c'$ remains $2$-regular and spanning, as remarked above, all $1$-components under $c'$ inside $C$ are finite cycles, whose union covers $C$. 

Also, since each $1$-component of $c$ is a double-ray, it must leave the finite set $P + \pseudogrid{g_1}{g_2}{N_3}{N_1}$ and hence meets some $R_q^\ell$ or $L_q^\ell$. Therefore, by construction each $1$-component of $c'$ inside $C$ meets some $R_q^\ell$ or $L_q^\ell$ and so, since $c'$ is standard outside of $P + \pseudogrid{g_1}{g_2}{N_0}{N_0}$ except at the squares $R_q^\ell$ or $L_q^\ell$, each such $1$-component meets both $P + (\pseudogrid{g_1}{g_2}{N_3}{N_1} \setminus \pseudogrid{g_1}{g_2}{N_2}{N_1})$ and $P + \pseudogrid{g_1}{g_2}{N_1}{N_1}$.

Moreover, all other colour components remain double-rays. This is clear for all $k$-components of $G$ if $k \neq 1,2$ (as the colours switchings of $(1,2)$-standard squares did not affect these other colours). However, it is also clear for the $1$-coloured double-rays outside of $C$ and also for all $2$-coloured components, as we chose our standard squares $R_q^\ell$ and $L_q^\ell$ `staggered', so as not to create any finite monochromatic cycles, see Figure~\ref{fig_capoff} (recall that every $x_\ell + \Delta$ is isomorphic to the grid). 

Finally, since $N_1 > N_0$, the edge set
\begin{align*}
\{(x_\ell & + n g_1 + N_1 g_2, x_\ell + (n+1) g_1 + N_1 g_2)\colon -N_3 \leq |n| < N_3-1\} \\
& \cup \Set{(v^\ell_1,v^\ell_1 + g_2),((w^\ell_1+g_1,w^\ell_1 + g_1 + g_2))} \\
& \cup \Set{(x_\ell + n g_1 + (N_1-1) g_2, x_\ell + (n+1) g_1 + (N_1-1) g_2)}:{-N_3 \leq n < -N_1} \\
& \cup \Set{(x_\ell + n g_1 + (N_1-1) g_2, x_\ell + (n+1) g_1 + (N_1-1) g_2)}:{N_1 \leq n < N_3}
\end{align*}
meets only $R^\ell_1$ and $L^\ell_1$ and therefore is easily seen to be part of the same $1$-component of $c'$. In Figure~\ref{fig_capoff}, these edges correspond to the red line at the top, and the two lines below it on either side of $x_\ell + \pseudogrid{g_1}{g_2}{N_1}{N_1}$.
\end{proof}

\subsection{Combining cycles inside each coset of $\Delta$}
\label{subsec_3}

In the previous step we chose the $(1,2)$-standard squares at which we performed colour switchings in a staggered manner in the grids $x_l + \pseudogrid{g_1}{g_2}{N_3}{N_1}$, so that we could guarantee that all the $2$-components were still double-rays afterwards. In later steps we will no longer be able to be as explicit about which standard squares we perform colour switchings at, and so we will require the following definitions to be able to say when it is `safe' to perform a colour switching at a standard square.

\begin{defn}[Crossing edges]
Suppose $R = \set{ (v_i,v_{i+1})}:{i \in \mathbb{Z}}$ is a double-ray and $e_1 = (v_{j_1},v_{j_2})$ and $e_2 = (v_{k_1},v_{k_2})$ are edges with $j_1 < j_2$ and $k_1 < k_2$. We say that $e_1$ and $e_2$ \emph{cross} on $R$ if either $j_1 < k_1 < j_2 < k_2$ or $k_1<j_1<k_2<j_2$.
\end{defn}

\begin{lemma}
\label{l:crossingwithstandard}
For an edge-colouring $c \colon E(G(\Gamma,S)) \rightarrow [s]$, suppose that $\stdsquare{x}{g_i}{g_k}$ is an $(i,k)$-standard square with $g_i \neq g_k$, and further that the two $k$-coloured edges $(x,x+g_k)$ and $(x+g_i,x+g_i+g_k)$ of $\stdsquare{x}{g_i}{g_k}$ lie on the same \emph{standard $k$-double-ray} $R = \doubleray{x}{g_k}$. Then the two $i$-coloured edges of $\stdsquare{x}{g_i}{g_k}$ cross on $R$.
\end{lemma}

\begin{proof}
Write $e_1 = (x,x+g_i)$ and $e_2 = (x+g_k,x+g_k+g_i)$ for the two $i$-coloured edges of $\stdsquare{x}{g_i}{g_k}$. The assumption that $(x,x+g_k)$ and $(x+g_i,x+g_i+g_k)$ both lie on $\doubleray{x}{g_k}$ implies that $g_i = rg_k$ for some $r \in \Z \setminus \singleton{-1,0,1}$. If $r > 1$, we have $x < x + g_k < x+g_i < x+g_k+g_i$ (where $<$ denotes the natural linear order on the vertex set of the double-ray), and if $r < -1$, we have $x + g_i < x + g_k + g_i < x < x+g_k$, and so the edges $e_1$ and $e_2$ indeed cross on $R$.
\end{proof}

\begin{defn}[Safe standard square]
Given an edge colouring $c \colon E(G(\Gamma,S)) \to [s]$ we say an $(i,k)$-standard square $\stdsquare{x}{g_i}{g_k}$
is \emph{safe} if $g_i \neq -g_k$ and either
\begin{itemize}
\item the $k$-components for $c$ meeting $T$ are distinct double-rays, or
\item there is a unique $k$-component for $c$ meeting $T$, which is a double-ray on which $(x,x+g_i)$ and $(x+g_k,x+g_i+g_k)$  cross.
\end{itemize}
\end{defn}

The following lemma tells us, amongst other things, that if we perform a colour switching at a safe $(1,k)$-standard square then the $k$-components in the resulting colouring meeting that square will still be double-rays. 

\begin{lemma}\label{l:easyflip2}
Let $c \colon E(G(\Gamma,S)) \to [s]$ be an edge colouring, T = $\stdsquare{x}{g_i}{g_k}$
be an $(i,k)$-standard square with $g_i \neq -g_k$, and $c'$ be the colouring obtained by performing a colour switching on $T$. Suppose that the $i$ and $k$-components for $c$ meeting $T$ are all $2$-regular, and that there are two distinct $i$-components $C_1$ and $C_2$ meeting $T$, at least one of which is a finite cycle. Then the following statements are true:
\begin{itemize}
\item There is a single $i$-component for $c'$ meeting $T$ which covers $V(C_1) \cup V(C_2)$;
\item If the $k$-components for $c$ meeting $T$ are distinct double-rays then the $k$-components for $c'$ meeting $T$ are distinct double-rays;
\item If there is a unique $k$-component for $c$ meeting $T$, which is a double-ray on which $(x,x+g_i)$ and $(x+g_k,x+g_i+g_k)$ cross, then there is unique $k$-component for $c'$ meeting $T$, which is a double-ray.
\end{itemize}
\end{lemma}

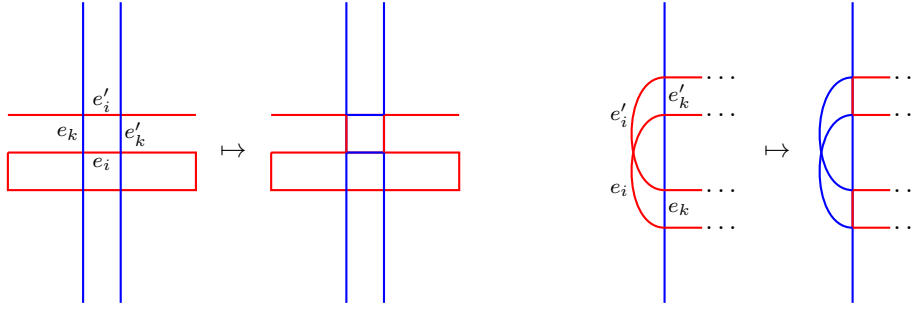
\begin{figure}[ht!]
\begin{center}

\begin{tikzpicture}[scale=0.5]

\draw[red,thick] (-2,0) -- (3,0) -- (3,-1) -- (-2,-1) -- (-2,0);
\draw[red,thick] (3,1) -- (-2,1);

\node at (.5,-.3) {\footnotesize{$e_i$}};
\node at (.5,1.4) {\footnotesize{$e'_i$}};
\node at (-.4,.5) {\footnotesize{$e_k$}};
\node at (1.4,.5) {\footnotesize{$e'_k$}};
  
\draw[blue,thick] (0,-4) -- (0,4);
\draw[blue,thick] (1,4) -- (1,-4);

\node  at (4,0) {$\mapsto$};

\begin{scope}[shift={(7,0)}]

\draw[red,thick] (-2,0) -- (3,0) -- (3,-1) -- (-2,-1) -- (-2,0);
\draw[red,thick] (3,1) -- (-2,1);
  
\draw[blue,thick] (0,-4) -- (0,4);
\draw[blue,thick] (1,4) -- (1,-4);
\draw[blue,thick] (0,0)--(1,0);
\draw[blue,thick] (0,1)--(1,1);
\draw[red, thick] (0,0)--(0,1);
\draw[red,thick] (1,0)--(1,1);

\end{scope}
\end{tikzpicture} \qquad \qquad \quad
\begin{tikzpicture}[scale=0.5]

\draw[blue,thick] (0,-4) -- (0,4);

\node at (.4,1.5) {\footnotesize{$e'_k$}};
\node at (.4,-1.5) {\footnotesize{$e_k$}};

\node at (-1.2,1) {\footnotesize{$e'_i$}};
\node at (-1.2,-1) {\footnotesize{$e_i$}};

\draw[red,thick] (0,-2)--(1,-2) (0,-1)--(1,-1) (0,1)--(1,1)  (0,2)--(1,2) ;
\node  at (1.5,-2) {$\ldots$};
\node  at (1.5,-1) {$\ldots$};
\node  at (1.5,1) {$\ldots$};
\node  at (1.5,2) {$\ldots$};

\draw[red,thick]    (0,-2) to[out=180,in=180] (0,1);
\draw[red,thick]    (0,-1) to[out=180,in=180] (0,2);

\node  at ( 3,0) {$\mapsto$};

\begin{scope}[shift={(5,0)}]
\draw[blue,thick] (0,-4) -- (0,4);

\draw[red,thick] (0,-2)--(1,-2) (0,-1)--(1,-1) (0,1)--(1,1)  (0,2)--(1,2) ;
\node  at (1.5,-2) {$\ldots$};
\node  at (1.5,-1) {$\ldots$};
\node  at (1.5,1) {$\ldots$};
\node  at (1.5,2) {$\ldots$};

\draw[red,thick] (0,-2)--(0,-1) (0,1)--(0,2);

\draw[blue,thick]    (0,-2) to[out=180,in=180] (0,1);
\draw[blue,thick]    (0,-1) to[out=180,in=180] (0,2);
\end{scope}
\end{tikzpicture}
\end{center}
\caption{The two situations in Lemma~\ref{l:easyflip2} with $i$ in red and $k$ in blue.}
\label{fig_colourflip}
\end{figure}

\begin{proof}

Let us write $e_i=(x,x+g_i)$, $e_k=(x,x+g_k)$, $e'_i=(x+g_k,x + g_i + g_k)$ and $e'_k=(x+g_i,x + g_i + g_k)$, so that $\stdsquare{x}{g_i}{g_j} = \Set{e_i,e_k,e'_i,e'_k}$.

For the first item, let the $i$-components for $c$ be $e_i \in C_1$ and $e'_i\in C_2$, where without loss of generality $C_2$ is a finite cycle. Then $C_2 - e'_i$ is a finite path, and $C_1 - e_i$ has at most $2$ components, one containing $x$ and one containing $x + g_i$. Hence, the $i$-component for $c'$ meeting $T$, $(C_1 \cup C_2) - \Set{e_i,e'_i} +  \Set{e_k,e'_k}$, is connected and covers $V(C_1) \cup V(C_2)$.

For the second item, let the $k$-components for $c$ be $e_k \in D_1$ and $e'_k\in D_2$. Then $D_1 - e_k$ has two components, a ray starting at $x$ and a ray starting at $x+g_k$. Similarly, $D_2 - e'_k$ has two components, a ray starting at $x+g_i$ and a ray starting at $x + g_i + g_k$. Hence, the $k$-components for $c'$ meeting $T$, which are the components of $(D_1 \cup D_2) -  \Set{e_k,e'_k} +  \Set{e_i,e'_i} $, are distinct double-rays.

Finally, if there is a single $k$-component $D$ for $c$ meeting $T$ such that $D$ is a double-ray, then $D - \Set{e_k,e'_k}$ consist of three components. Since $e_i$ and $e'_i$ cross on $D$ there are two cases as to what these components are. Either the components consist of two rays, starting at $x$ and $x + g_i + g_k$ and a finite path from $x+g_k$ to $x+g_i$, or the components consist of two rays, starting at $x+g_i$ and $x+g_k$, and a finite path from $x+g_i+g_k$ to $x$. In either case, the $k$-component for $c'$ meeting $T$, namely $D -  \Set{e_k,e'_k} +  \Set{e_i,e'_i} $, is a double-ray.
\end{proof}

Lemma \ref{l:easyflip2} is also useful as the first item allows us to use $(1,k)$ colour switchings to combine two $1$-components into a single $1$-component which covers the same vertex set.

\begin{step}[Combining cycles step]
\label{step_cc}
We can change $c'$ from Step~\ref{step_co} via colour switchings of finitely many $(1,2)$-standard squares to a colouring $c''$ satisfying
\begin{itemize}
\item $c''=c'=c$ on $E(G[X])$;
\item every $1$-component in $c''$ meeting $P + \pseudogrid{g_1}{g_2}{N_2}{N_1}$ is a finite cycle intersecting both $P + (\pseudogrid{g_1}{g_2}{N_3}{N_1} \setminus \pseudogrid{g_1}{g_2}{N_2}{N_1})$ and $P + \pseudogrid{g_1}{g_2}{N_1}{N_1}$;
\item every other $1$-component, and all other components of all other colour classes of $c''$ are double-rays;
\item every $1$-component in $c''$ meeting some $x_k + (\pseudogrid{g_1}{g_2}{N_2}{N_1} \setminus \pseudogrid{g_1}{g_2}{N_0}{N_0}) $ covers $x_k + (\pseudogrid{g_1}{g_2}{N_2}{N_1} \setminus \pseudogrid{g_1}{g_2}{N_0}{N_0})$;
\item $c''$ is standard outside of $P + \pseudogrid{g_1}{g_2}{N_3}{N_1}$ and inside of $P + (\pseudogrid{g_1}{g_2}{N_1}{N_1} \setminus \pseudogrid{g_1}{g_2}{N_0}{N_0})$.
\end{itemize}
\end{step}
\begin{proof}
Our plan will be to go through the `grids' $x_k + \pseudogrid{g_1}{g_2}{N_2}{N_1}$ in order, from $k=0$ to $t$, and use colour switchings to combine all the $1$-components which meet $x_k + (\pseudogrid{g_1}{g_2}{N_2}{N_1} \setminus \pseudogrid{g_1}{g_2}{N_0}{N_0}) $ into a single $1$-component. We note that, since $c'$ is not standard on $X$, it may be the case that these $1$-components also meet $x_{k'} + \pseudogrid{g_1}{g_2}{N_2}{N_1}$ for $k' \neq k$.

We claim inductively that there exists a sequence of colourings $c' = c_0,c_1, \ldots, c_t = c''$ such that for each $0 \leq \ell \leq t$:

\begin{itemize}
\item $c_\ell=c'=c$ on $E(G[X])$;
\item every $1$-component in $c_\ell$ meeting $P + \pseudogrid{g_1}{g_2}{N_2}{N_1}$ is a finite cycle intersecting both $P + (\pseudogrid{g_1}{g_2}{N_3}{N_1} \setminus \pseudogrid{g_1}{g_2}{N_2}{N_1})$ and $P + \pseudogrid{g_1}{g_2}{N_1}{N_1}$;
\item for every $k \leq \ell$, every $1$-component in $c_\ell$ meeting $x_k + (\pseudogrid{g_1}{g_2}{N_2}{N_1} \setminus \pseudogrid{g_1}{g_2}{N_0}{N_0})$ covers $x_k + (\pseudogrid{g_1}{g_2}{N_2}{N_1} \setminus \pseudogrid{g_1}{g_2}{N_0}{N_0})$;
\item for every $k > \ell$, $c_\ell=c'$ on $x_k +\pseudogrid{g_1}{g_2}{N_2}{N_1}$
\item every other $1$-component, and all other components of all other colour classes of $c_\ell$ are double-rays;
\item $c_\ell$ is standard outside of $P + \pseudogrid{g_1}{g_2}{N_3}{N_1}$ and inside of $P + (\pseudogrid{g_1}{g_2}{N_1}{N_1} \setminus \pseudogrid{g_1}{g_2}{N_0}{N_0})$.
\end{itemize}

In Step~\ref{step_co} we constructed $c_0=c'$ such that this holds. Suppose that $0 < \ell \leq t$, and that we have already constructed $c_k$ for $k < \ell$.
 
For $q  \in [4N_1 - 2]$ we define $T_q = \stdsquare{v_q}{g_1}{g_2}$
to be the $(1,2)$-square with base point 
\[v_q = \begin{cases} x_\ell + (N_2 +2 - 2q)g_1 + (N_1 - q)g_2 & \text{ if } \; q \leq 2N_1 - 1, \; \text{ and} \\
					x_\ell - (N_2 + 3- 2q')g_1 + (N_1 - q')g_2 & \text{ if } \; q'= q-(2N_1-1) \geq 1.
\end{cases}
\]
With these definitions, $T_{2N_1-1+q}$ is the mirror image of $T_q$ for all $q \in [2N_1-1]$ along the $y$-axis. Moreover, since $N_2 \geq 5N_1$, each $T_q$ is contained within $x_k + (\pseudogrid{g_1}{g_2}{N_2}{N_1} \setminus \pseudogrid{g_1}{g_2}{N_1}{N_1})$.

We will combine the $1$-components in $c_{\ell-1}$ which meet $x_\ell + (\pseudogrid{g_1}{g_2}{N_2}{N_1} \setminus \pseudogrid{g_1}{g_2}{N_0}{N_0}) $ into a single component by performing colour switchings at some of the $(1,2)$-squares $T_q$. Let us show first that most of the induction hypotheses are maintained regardless of the subset of the $T_q$ we make switchings at.

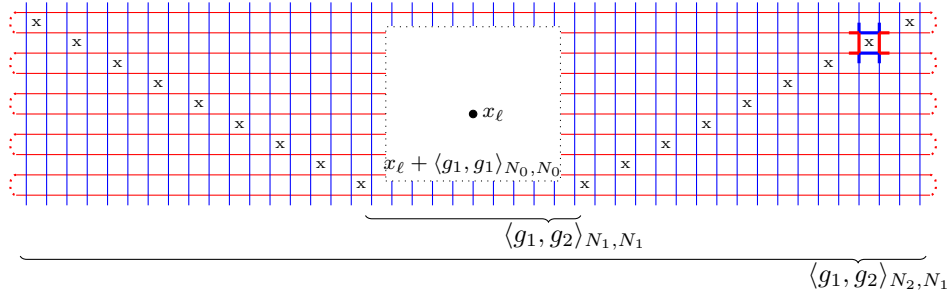
\begin{figure}[ht!]
\centering
\begin{tikzpicture}[scale=.27]

\def \NZero {4}
\def \NOne {5}
\def \NThree {22} 

\pgfmathsetmacro\rows{2*\NOne} 
\foreach \n in {1,...,\rows}
{
\pgfmathsetmacro\y{\n-\NOne}
\draw[red, thin] (-\NThree-.5,\y) -- (\NThree+.5,\y);
}

\foreach \n in {-\NThree,...,\NThree}
{
\draw[blue, thin] (\n,-\NOne+.5) -- (\n,\NOne+.5);
}

\draw[mybrace=0.15] (\NOne+.3,-\NOne) -- (-\NOne-.3,-\NOne);
\node at (\NOne,-\NOne-1) {$\pseudogrid{g_1}{g_2}{N_1}{N_1}$};

\draw[mybrace=0.075] (\NThree+.3,-\NOne-2) -- (-\NThree-.3,-\NOne-2);
\node at (\NThree-2,-\NOne-3) {$\pseudogrid{g_1}{g_2}{N_2}{N_1}$};

\draw [fill=white, dotted] (-\NZero-.3,-\NZero+.7) rectangle (\NZero+.3,\NZero+.3);
\node at (0,-\NZero+1.5) {\footnotesize{$x_\ell + \pseudogrid{g_1}{g_1}{N_0}{N_0}$}};
\node[draw, circle,scale=.3, fill] (xl) at (0,0) {};
\node at (1,0) {\footnotesize{$x_\ell$}};

\foreach \n in {1,...,\NOne}
{
\draw[thick, red, dotted] (-\NThree-.5, 7 - 2*\n) to[out=180,in=180] (-\NThree-.5,6 - 2*\n);
\draw[thick, red, dotted] (\NThree+.5, 7 - 2*\n) to[out=0,in=0] (\NThree+.5,6 - 2*\n);
}
\pgfmathsetmacro\rows{\NOne-1}
\pgfmathsetmacro\rowsupper{\NOne-1}
\foreach \q in {-\rows,...,\rowsupper}
{
\pgfmathsetmacro\x{\NThree - \NOne + 2*\q - 4}
\pgfmathsetmacro\y{ \q }
\node at (\x+.5,\y+.5) {$\textnormal{\tiny{x}}$};
}
\foreach \q in {-\rows,...,\rowsupper}
{
\pgfmathsetmacro\x{-(\NThree - \NOne + 2*\q)+3}
\pgfmathsetmacro\y{-\NOne + 5 + \q}
\node at (\x+.5,\y+.5) {$\textnormal{\tiny{x}}$};
}

\def \NThree {20} 
\draw [blue,very thick] (\NThree,\NOne-1) -- (\NThree -1,\NOne-1);
\draw [blue,very thick] (\NThree,\NOne-1) -- (\NThree,\NOne-0.5);
\draw [blue,very thick] (\NThree -1,\NOne-1) -- (\NThree -1,\NOne-0.5);

\draw [blue,very thick] (\NThree,\NOne-2) -- (\NThree -1,\NOne-2);
\draw [blue,very thick] (\NThree,\NOne-2) -- (\NThree,\NOne-2.5);
\draw [blue,very thick] (\NThree -1,\NOne-2) -- (\NThree -1,\NOne-2.5);

\draw [red,very thick] (\NThree,\NOne-1) -- (\NThree,\NOne-2);
\draw [red,very thick] (\NThree,\NOne-1) -- (\NThree+.5,\NOne-1);
\draw [red,very thick] (\NThree,\NOne-2) -- (\NThree+.5,\NOne-2);

\draw [red,very thick] (\NThree -1,\NOne-1) -- (\NThree -1,\NOne-2);
\draw [red,very thick] (\NThree -1,\NOne-1) -- (\NThree -1.5,\NOne-1);
\draw [red,very thick] (\NThree -1,\NOne-2) -- (\NThree -1.5,\NOne-2);

\end{tikzpicture}
\caption{The standard squares $T_q$, with a colour switching performed at $T_{2}$.}
\label{fig_comcyc}
\end{figure}

We note that, since $c_{\ell-1}$ is standard inside of $x_\ell + (\pseudogrid{g_1}{g_1}{N_2}{N_1} \setminus \pseudogrid{g_1}{g_2}{N_0}{N_0})$ and outside of $P + \pseudogrid{g_1}{g_2}{N_3}{N_1}$, and $g_1 \neq -g_2$, each $T_q$ is a safe $(1,2)$-standard square for $c_{\ell-1}$. Furthermore, by construction, even if we perform colour switchings at any subset of the $T_q$, the remaining squares remain standard and safe.

Hence, by Lemma \ref{l:easyflip2} and the induction assumption, after performing colour switchings at any subset of the standard squares $T_q$ all $2$-components of the resulting colouring will be double-rays. Secondly, these colour switching will not change the colouring outside of $P + \pseudogrid{g_1}{g_2}{N_2}{N_1}$ and inside of $P + \pseudogrid{g_1}{g_2}{N_1}{N_1}$, or in any $x_k + \pseudogrid{g_1}{g_2}{N_2}{N_1}$ with $k \neq \ell$. In particular, every $1$-component not meeting $P + \pseudogrid{g_1}{g_2}{N_2}{N_1}$ will still be a double-ray. Finally, again by Lemma \ref{l:easyflip2}, every $1$-component of the resulting colouring meeting $P + \pseudogrid{g_1}{g_2}{N_2}{N_1}$ will be a finite cycle which covers the vertex set of some union of $1$-components in $c_{\ell-1}$, and hence will intersect both $P + (\pseudogrid{g_1}{g_2}{N_3}{N_1} \setminus \pseudogrid{g_1}{g_2}{N_2}{N_1})$ and $P + \pseudogrid{g_1}{g_2}{N_1}{N_1}$.

Let us write $e_q = (v_q,v_q + g_1)$ for each $q \in [4N_1-2]$.
Since $c_{\ell-1} = c'$ on $x_\ell +\pseudogrid{g_1}{g_2}{N_2}{N_1}$, and by Step~\ref{step_co} $c'$ is standard on $x_\ell + (\pseudogrid{g_1}{g_2}{N_2}{N_1} \setminus \pseudogrid{g_1}{g_2}{N_0}{N_0})$, each $1$-component of $c_{\ell-1}$ that meets $x_\ell + (\pseudogrid{g_1}{g_2}{N_2}{N_1} \setminus \pseudogrid{g_1}{g_2}{N_0}{N_0})$ contains at least one $e_q$. Also, $e_1$ and $e_{2N_1}$ belong to the same $1$-component by the last claim in Step~\ref{step_co}. Let us write $\script{C}$ for the collection of such cycles, and consider the map
$$\alpha \colon \script{C} \to \Set{1, \ldots, 4N_1-1}, \; C \mapsto \min \set{q}:{e_q \in E(C)},$$
which maps each cycle to the first $e_q$ that it contains. Since $\script{C}$ is a disjoint collection of cycles, the map $\alpha$ is injective. Now let $c_\ell$ be the colouring obtained from $c_{\ell-1}$ by switching all standard squares in
$$\script{T}=\set{T_q}:{q \in \operatorname{ran}(\alpha)} \setminus \{T_1\}.$$

We claim that $c_\ell$ satisfies our induction hypothesis for $\ell$. By the previous comments it will be sufficient to show

\begin{claim}
Every $1$-component in $c_\ell$ meeting $x_\ell + (\pseudogrid{g_1}{g_2}{N_2}{N_1} \setminus \pseudogrid{g_1}{g_2}{N_0}{N_0})$ covers $x_\ell + (\pseudogrid{g_1}{g_2}{N_2}{N_1} \setminus \pseudogrid{g_1}{g_2}{N_0}{N_0}) $.
\end{claim}

To see this, we index $\script{C} = \Set{C_1, \ldots, C_r}$ such that $u< v$ implies $\alpha(C_u) < \alpha(C_v)$, and consider the sequence of colourings $\set{c^z}:{ z\in [r]}$ where $c^1 = c_\ell$ and each $c^z$ is obtained from $c^{z-1}$ by switching the standard square $T_{\alpha(C_z)}$.

Let us show by induction that for every $z \in [r]$ there is an $1$-component of $c^z$ which covers $\bigcup_{y \leq z} C_y$. For $z=1$ the claim is clearly true. So, suppose $z > 1$. Since $\alpha(C_z)$ is minimal in $\{ \alpha(C_y) \colon y \geq z \}$ it follows that $e_q \in \bigcup_{y < z} C_y$ for every $q < \alpha(C_z)$. Note that, since $c_{\ell-1} = c'$ on $x_\ell + \pseudogrid{g_1}{g_2}{N_2}{N_1}$, it follows from the final claim in the Cap-off step that $C_1$ contains both $e_1$ and $e_{2N_1}$, and so $\alpha(C_z) \neq 2N_1$.

Consider the standard square $T_{\alpha(C_z)}$. Since $c_{\ell-1} = c'$ on $x_\ell + \pseudogrid{g_1}{g_2}{N_2}{N_1}$, by construction the edge `opposite' to $e_{\alpha(C_z)}$ in $T_{\alpha(C_z)}$, that is, $e_{\alpha(C_z)} + g_j$, is in the same $1$-component in $c_{\ell-1}$ as $e_{\alpha(C_z) - 1}$, and hence is contained in $\bigcup_{y < z} C_y$.

Therefore, by Lemma \ref{l:easyflip2}, after performing an $(1,2)$-colour switching at $T_{\alpha(C_z)}$, the $1$-component in $c^z$ contains $\bigcup_{y \leq z} C_y$.

Hence, there is an $1$-component of $c_\ell = c^r$ which covers $\bigcup_{y \leq r} C_y$, and so there is a unique $1$-component of $c_\ell$ meeting $x_\ell + (\pseudogrid{g_1}{g_2}{N_2}{N_1} \setminus \pseudogrid{g_1}{g_2}{N_0}{N_0}) $ which covers it, establishing the claim.
\end{proof}

\subsection{Combining cycles across different cosets of $\Delta$}
\label{subsec_4}
In the third and final step we join the finite cycles covering each $x_\ell + (\pseudogrid{g_1}{g_2}{N_1}{N_1} \setminus \pseudogrid{g_1}{g_2}{N_0}{N_0})$ into a single finite cycle, and then make one final switch to absorb this cycle into a double-ray. The resulting colouring will then satisfy the conditions of Lemma \ref{lem_mainlemma}.

\begin{step}[Combining cosets step]
\label{step_mainlemma_quant}
We can change $c''$ from the previous lemma to an almost-standard colouring $\hat{c}$ such that 
\begin{itemize}
\item $\hat{c}=c''=c'=c$ on $E(G[X])$;
\item Some component in colour $1$ covers $P + \pseudogrid{g_1}{g_2}{N_1}{N_1}$.
\end{itemize}
\end{step}

\begin{proof}
Recall that $P = \Set{x_0, \ldots, x_t}$ is such that $P^\Delta = \Set{x_0 + \Delta, \ldots, x_t + \Delta}$ is a finite, graph-theoretic path in the Cayley graph of the quotient $\Gamma / \Delta$ with generating set $S \setminus \{g_1,g_2\}$. Moreover, recall from Section~\ref{subsec_1.5} that $N_1 > N_0$ was chosen so that for the initial colouring $c$ there were $t^2$ many disjoint standard double-rays
\[ \script{R} = \set{R^k_\ell}:{1 \leq k,\ell \leq t} \]
such that for every $\ell$, the double-rays in $\set{R^k_\ell = \doubleray{y^k_\ell}{g_{n(\ell)}}}:{k \in [t]}$ are standard $n(\ell)$-double-rays containing an edge
$$e^k_\ell=(y^k_\ell,y^k_\ell + g_{n(\ell)}) \in E(R^k_\ell) \cap E(x_{\ell-1}+\Delta, x_\ell + \Delta)$$
so that all
$T^k_\ell = \stdsquare{y^k_\ell}{g_1}{g_{n(\ell)}}$
are edge-disjoint $(1,n(\ell))$-standard squares for the colouring $c$ contained in the subgraph induced by $P + \pseudogrid{g_1}{g_2}{N_1-3}{N_1-3}$ which have empty intersection with $\Set{x_{\ell-1},x_\ell} + \pseudogrid{g_1}{g_2}{N_0}{N_0}$. However, since we only altered the $(1,2)$-subgraphs of $G$ in Step~\ref{step_co} and \ref{step_cc}, it is clear that all these standard double-rays and standard squares for $c$ remain standard also for the colourings $c'$ and in particular $c''$.

\begin{figure}[ht!]
\centering
\begin{tikzpicture}[scale=.2]

\def \NZero {6}
\def \NOne {10}	
\def \NThree {10}

\begin{scope}[x={(0.8cm,0.4cm)},y={(0,1)}] 
\pgfmathsetmacro\rows{2*\NOne} 
\foreach \n in {1,...,\rows}
{
\pgfmathsetmacro\y{\n-\NOne}
\draw[red, thin] (-\NThree-.5,\y) -- (\NThree+.5,\y);
}

\foreach \n in {-\NThree,...,\NThree}
{
\draw[blue, thin] (\n,-\NOne+.5) -- (\n,\NOne+.5);
}

\draw[mybrace=0.5] (\NOne+.3,-\NOne) -- (-\NOne-.3,-\NOne);
\node[rotate=25] at (0,-\NOne-2) {$ x_0 + \pseudogrid{g_1}{g_2}{N_1}{N_1}$};w

\draw [fill=white, dotted] (-\NZero-.3,-\NZero+.7) -- (\NZero+.3,-\NZero+.7)  -- (\NZero+.3,\NZero+.3) -- (-\NZero-.3,\NZero+.3);
\node[scale=0.9,rotate=25] at (0,-\NZero+1.5) {\footnotesize{$x_0 + \pseudogrid{g_1}{g_2}{N_0}{N_0}$}};
\node[draw, circle,scale=.3, fill] (xl) at (0,0) {};
\node at (0,1) {\footnotesize{$x_0$}};

\node at (\NOne-1,\NOne-1) (gn1start) {};
\node at (\NOne-2,\NOne-1) (gn1start2) {};
\end{scope}

\begin{scope}[shift={(20,0)},x={(0.8cm,0.4cm)},y={(0,1)}]

\pgfmathsetmacro\rows{2*\NOne} 
\foreach \n in {1,...,\rows}
{
\pgfmathsetmacro\y{\n-\NOne}
\draw[red, thin] (-\NThree-.5,\y) -- (\NThree+.5,\y);
}

\foreach \n in {-\NThree,...,\NThree}
{
\draw[blue, thin] (\n,-\NOne+.5) -- (\n,\NOne+.5);
}

\draw[mybrace=0.5] (\NOne+.3,-\NOne) -- (-\NOne-.3,-\NOne);
\node[rotate=25] at (0,-\NOne-2) {$x_1 + \pseudogrid{g_1}{g_2}{N_1}{N_1}$};

\draw [fill=white, dotted] (-\NZero-.3,-\NZero+.7) -- (\NZero+.3,-\NZero+.7)  -- (\NZero+.3,\NZero+.3) -- (-\NZero-.3,\NZero+.3);
\node[scale=0.9,rotate=25] at (0,-\NZero+1.5) {\footnotesize{$x_1 + \pseudogrid{g_1}{g_2}{N_0}{N_0}$}};
\node[draw, circle,scale=.3, fill] (xl) at (0,0) {};
\node at (0,1) {\footnotesize{$x_1$}};

\node at (\NOne-1,\NOne-1) (gn1end) {};
\node at (\NOne-2,\NOne-1) (gn1end2) {};

\node at (\NOne-1,\NOne-5) (gn2start) {};
\node at (\NOne-2,\NOne-5) (gn2start2) {};
\end{scope}

\begin{scope}[shift={(40,0)},x={(0.8cm,0.4cm)},y={(0,1)}]
\pgfmathsetmacro\rows{2*\NOne} 
\foreach \n in {1,...,\rows}
{
\pgfmathsetmacro\y{\n-\NOne}
\draw[red, thin] (-\NThree-.5,\y) -- (\NThree+.5,\y);
}

\foreach \n in {-\NThree,...,\NThree}
{
\draw[blue, thin] (\n,-\NOne+.5) -- (\n,\NOne+.5);
}

\draw[mybrace=0.5] (\NOne+.3,-\NOne) -- (-\NOne-.3,-\NOne);
\node[rotate=25] at (0,-\NOne-2) {$x_2 + \pseudogrid{g_1}{g_2}{N_1}{N_1}$};

\draw [fill=white, dotted] (-\NZero-.3,-\NZero+.7) -- (\NZero+.3,-\NZero+.7)  -- (\NZero+.3,\NZero+.3) -- (-\NZero-.3,\NZero+.3);
\node[scale=0.9,rotate=25] at (0,-\NZero+1.5) {\footnotesize{$x_2  + \pseudogrid{g_1}{g_2}{N_0}{N_0}$}};
\node[draw, circle,scale=.3, fill] (xl) at (0,0) {};
\node at (0,1) {\footnotesize{$x_2$}};

\node at (\NOne-1,\NOne-5) (gn2end) {};
\node at (\NOne-2,\NOne-5) (gn2end2) {};
\end{scope}

\begin{scope}[x={(0.8cm,0.4cm)},y={(0,1)}]
\draw[green] (gn1start) to[out=80,in=100] (gn1end);
\draw[green] (gn1start2) to[out=90,in=90] (gn1end2);
\node at (\NOne+11,9) {\footnotesize{$g_{n(1)}$}};
\end{scope}

\draw[green] (gn2start) to[out=80,in=100] (gn2end);
\draw[green] (gn2start2) to[out=90,in=90] (gn2end2);
\node at (\NOne+26,17) {\footnotesize{$g_{n(2)}$}};

\node at (53,0) {$\ldots$};

\end{tikzpicture}
\caption{Using $(1,n(\ell))$-standard squares to join up different cosets. For this picture, we assume wlog that $x_{\ell+1} = x_\ell + g_{n(\ell+1)}$.}
\label{fig_comcos}
\end{figure}
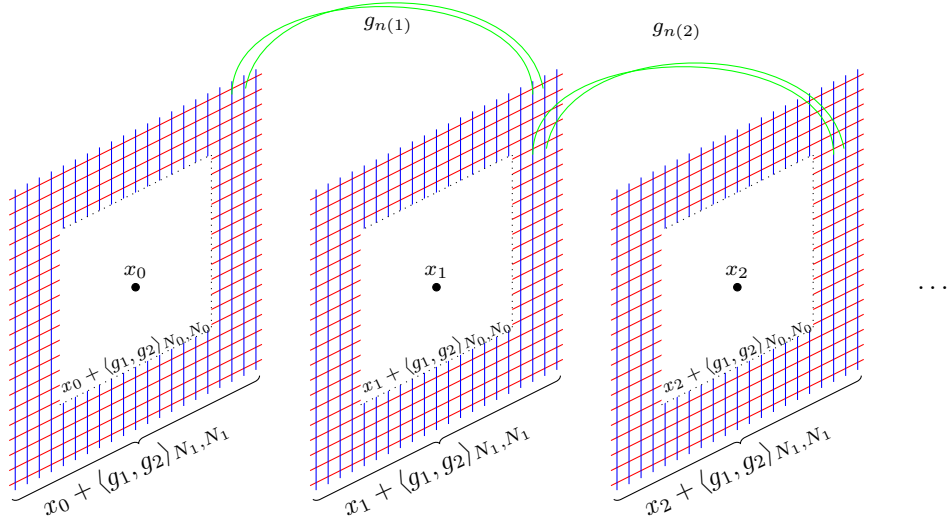

We claim that there exists a function $k \colon [t] \to [t] \cup \Set{\bot}$ such that iteratively switching $T^{k(\ell)}_\ell$ (or not doing anything at all if $k(\ell) = \bot$) results in a sequence of colourings $c'' = c_0,c_1, \ldots, c_t$ such that for each $0 \leq \ell \leq t$,
\begin{enumerate}
\item\label{it1} a single finite $1$-component in $c_\ell$ covers $\Set{x_0, \ldots, x_\ell} + (\pseudogrid{g_1}{g_2}{N_1}{N_1} \setminus \pseudogrid{g_1}{g_2}{N_0}{N_0})$,
\item\label{it2} for every $k$, every $1$-component in $c_\ell$ meeting $x_k + (\pseudogrid{g_1}{g_2}{N_1}{N_1} \setminus \pseudogrid{g_1}{g_2}{N_0}{N_0})$ is a finite cycle covering $x_k + (\pseudogrid{g_1}{g_2}{N_1}{N_1} \setminus \pseudogrid{g_1}{g_2}{N_0}{N_0})$, and
\item\label{it3} every other $1$-component, and all other components of all other colour classes in $c_\ell$ are double-rays.
\end{enumerate}

In Step~\ref{step_cc} we constructed a colouring $c_0=c''$ for which properties (1)--(3) are satisfied. Now suppose that $\ell \geq 1$, and that the colouring $c_{\ell-1}$ obtained by switching the standard squares $\set{T^{k(\ell')}_{\ell'}}:{\ell' \in [\ell-1]}$ satisfies (1)--(3).
By construction, each such standard square $T^{k(\ell')}_{\ell'}$ is incident with the ray $R^{k(\ell')}_{\ell'}$ and potentially one further $n(\ell')$-component. But since we had reserved more that $\ell-1$ different rays $R^1_\ell, \ldots, R^t_\ell$, it follows that some ray $R^{k(\ell)}_\ell$ remains a standard $n(\ell)$-coloured component for $c_{\ell-1}$. 

Both edges $(y^{k(\ell)}_\ell,y^{k(\ell)}_\ell +g_i)$ and $(y^{k(\ell)}_\ell+g_{n(\ell)},y^{k(\ell)}_\ell+g_{n(\ell)}+g_i)$ of $T^{k(\ell)}_\ell$
are contained in 
$\Set{x_{\ell-1},x_\ell} + (\pseudogrid{g_1}{g_2}{N_1}{N_1} \setminus \pseudogrid{g_1}{g_2}{N_0}{N_0})$, and hence are, by assumption (\ref{it2}), covered by finite $1$-cycles in $c_{\ell-1}$. If both edges lie in the same finite $1$-cycle, there is nothing to do (and we redefine $k(\ell) := \bot$, and let $c_\ell= c_{\ell-1}$). However, if they lie on different finite cycles, we perform a colour switching on the standard square $T^{k(\ell)}_\ell$ and claim that the resulting $c_\ell$ is as required. By  Lemma~\ref{l:easyflip2}, the two finite $1$-components merge into a single finite cycle, and so (\ref{it1}) and (\ref{it2}) are certainly satisfied for $c_\ell$. 

To see (\ref{it3}), we need to verify that $T^{k(\ell)}_\ell$ is, when we perform the switching, safe. However, $T^{k(\ell)}_\ell$ was chosen so that the edge $(y^{k(\ell)}_\ell,y^{k(\ell)}_\ell + g_{n(\ell)})\in T^{k(\ell)}_\ell$ lies on a standard double-ray $R=R^{k(\ell)}_\ell$ of $c_{\ell-1}$. Also, by the inductive assumption (\ref{it3}), the second $n(\ell)$-coloured edge $(y^{k(\ell)}_\ell +g_i,y^{k(\ell)}_\ell+g_i+g_{n(\ell)}) \in T^{k(\ell)}_\ell$ lies on an $n(l)$-coloured double-ray $R'$ in $c_{\ell-1}$. If $R$ and $R'$ are distinct, then $T^{k(\ell)}_\ell$ is safe, and if $R=R'$ then, since $R$ is a standard $n(\ell)$-double-ray, Lemma~\ref{l:crossingwithstandard} implies that $T^{k(\ell)}_\ell$ is safe. Hence $c_\ell$ satisfies (\ref{it3}). This completes the induction step.

Thus, by (\ref{it1}) and (\ref{it3}), we obtain an edge-colouring $c_t$ for $G$ such that a single finite $1$-component covers $P + (\pseudogrid{g_1}{g_2}{N_1}{N_1}\setminus \pseudogrid{g_1}{g_2}{N_0}{N_0})$, and all other $1$-components and all other components of other colour classes in $c_t$ are double-rays. Furthermore, since every $1$-component which meets $P+\pseudogrid{g_1}{g_2}{N_0}{N_0}$ must meet $P + (\pseudogrid{g_1}{g_2}{N_1}{N_1}\setminus \pseudogrid{g_1}{g_2}{N_0}{N_0})$, it follows that the $1$-component in fact covers $P+\pseudogrid{g_1}{g_2}{N_0}{N_0}$. Moreover, since $T^{k(\ell)}_\ell \subset P + \pseudogrid{g_1}{g_2}{N_1-3}{N_1-3}$ for all $\ell \in [t]$, it follows that $c_t$ is standard on $x_0 + \p{\pseudogrid{g_1}{g_2}{N_1}{\infty} \setminus \pseudogrid{g_1}{g_2}{N_1-3}{N_1-3}}$, and that it is standard outside of $P + \pseudogrid{g_1}{g_2}{N_3}{N_1}$. Hence, the square $\stdsquare{x}{g_1}{g_2}$ with base point $x = x_0 + (N_1-2)g_1 + N_1 g_2$ is a standard $(1,2)$-square such that 
\begin{itemize}
\item the edge $(x,x+g_1)$ lies on the finite $1$-cycle of $c_t$,
\item the edge $(x+g_2,x+g_2+g_1)$ lies on standard $1$-double-ray $\doubleray{x+g_2}{g_1}$ (lying completely outside of $P + \pseudogrid{g_1}{g_2}{N_3}{N_1}$) of $c_t$, and
\item the edges $(x,x+g_2)$ and $(x+g_1,x+g_2+g_1)$ lie on distinct standard $2$-double-rays $\doubleray{x}{g_2}$ and $\doubleray{x+g_1}{g_2} \subseteq x_0 + \p{\pseudogrid{g_1}{g_2}{N_1}{\infty} \setminus \pseudogrid{g_1}{g_2}{N_1-3}{N_1-3}}$.
\end{itemize}
Therefore, we may perform a colour switching on $\stdsquare{x}{g_1}{g_2}$, which results, by Lemma~\ref{l:easyflip2}, in an almost standard colouring of $G$ such that a single $1$-component covers $P + \pseudogrid{g_1}{g_2}{N_1}{N_1}$, and hence $X$.
\end{proof}

\section{Hamiltonian decompositions of products}
\label{sec_products}

The techniques from the previous section can also be applied to give us the following general result about Hamiltonian decompositions of products of graphs. 

\main*

\begin{proof}
Suppose that $\Set{{R_i}\colon{i \in I}}$ and $\Set{{S_j}\colon{j \in J}}$ form decompositions of $G$ and $H$ into edge-disjoint Hamiltonian double-rays, where $I,J$ may be finite or countably infinite. Note that, for each $i \in I, j \in J$, $R_i \square S_j$ is a spanning subgraph of $G \square H$, and is isomorphic to the Cayley graph of $(\Z^2,+)$ with the standard generating set.

Let $\pi_G \colon G \square H \to G$ and $\pi_H \colon G \square H \to H$ the projection maps from $G \square H$ onto the respective coordinates.
As our \emph{standard colouring} for $G \square H$ we take the map
\[ c \colon E(G \square H) \to I \dot\cup J, \; e \mapsto \begin{cases}
i & \text{ if } e \in \pi_G^{-1}(E(R_i)), \\
j & \text{ if } e \in \pi_H^{-1}(E(S_j)).
\end{cases}
\]
Then each $R_i \square S_j$ is $2$-coloured (with colours $i$ and $j$), and this colouring agrees with the standard colouring of $C_{\Z^2} = G((\Z^2,+), \{(1,0),(0,1)\})$ from Section~\ref{sec_cov}.

We may suppose that $V(G) = \N = V(H)$. Fix a surjection $f \colon \N \to I \cup J$ such that every colour appears infinitely often.

By starting with $c_0 = c$ and applying Lemma~\ref{lem_mainlemma} recursively inside the spanning subgraphs $R_{f(k)} \square S_1$, if $f(k) \in I$, or inside $R_1 \square S_{f(k)}$, for $f(k) \in J$, we find a sequence of edge-colourings $c_k \colon G \square H \to I \cup J$ and natural numbers $M_k \leq N_k < M_{k+1}$ such that
\begin{itemize}
\item $c_{k+1}$ agrees with $c_k$ on the subgraph of $G \square H$ induced by $[0,M_{k+1}]^2$,
\item there is a finite path $D_k$ of colour $f(k)$ in $c_k$ covering $[0,N_k]^2$, and
\item $M_{k+1}$ is large enough such that $D_k \subset [0,M_{k+1}]^2$.
\end{itemize}
To be precise, suppose we already have a finite path $D_k$ of colour $f(k)$ in $c_k$ covering $[0,N_k]^2$, and at stage $k+1$ we have say $f(k+1) \in I$, and so we are considering $R_{f(k+1)} \square S_1 \cong C_{\Z^2}$. We choose
\begin{itemize}
\item $M_{k+1} > N_k$ large enough such that $D_k \subset [0,M_{k+1}]^2 \subset G\square H$, and
\item $N_{k+1} > M_{k+1}$ large enough such that $Q_1=[0,N_{k+1}]^2 \subset G\square H$ contains all edges where $c_k$ differs from the standard colouring $c$.
\end{itemize}
Next, consider an isomorphism $h\colon R_{f(k+1)} \square S_1 \cong C_{\Z^2}$. Pick a `square' $Q_2 \subset R_{f(k+1)} \square S_1$ with $Q_1 \subset Q_2$, i.e.\ a set $Q_2$ such that $h$ restricted to $Q_2$ is an isomorphism to the subgraph of $C_{\Z^2}$ induced by  $[-\tilde{N}_{k+1},\tilde{N}_{k+1}]^2 \subseteq \Z^2$ for some $\tilde{N}_{k+1} \in \N$, and then apply Lemma~\ref{lem_mainlemma} to $R_{f(k+1)} \square S_1$ and $Q_2$ to obtain a finite path $D_{k+1}$ of colour $f(k+1)$ in $c_{k+1}$ covering $Q_2$.

It follows that the double-rays $\Set{{T_i}\colon{i\in I}} \cup \Set{{T_j}\colon{j\in J}}$ with $T_\ell = \bigcup_{k \in f^{-1}(\ell) } D_k$ give the desired decomposition of $G \square H$. 
\end{proof}

\section{Open Problems}\label{sec_open}
As mentioned in Section \ref{sec_group}, the finitely generated abelian groups can be classified as the groups $\mathbb{Z}^n \oplus \bigoplus_{i=1}^r \mathbb{Z}_{q_i}$, where $n,r,q_1,\ldots,q_r \in \Z$. Theorem \ref{t:ZN} shows that Alspach's conjecture holds for every such group with $n \geq 2$, as long as each generator has infinite order. The question however remains as to what can be said about Cayley graphs $G(\Gamma,S)$ when $S$ contains elements of finite order.

\begin{prob}
Let $\Gamma$ be an infinite, finitely-generated, one-ended abelian group and $S$ be a generating set for $\Gamma$ which contains elements of finite order. Show that $G(\Gamma,S)$ has a Hamilton decomposition.
\end{prob}

Alspach's conjecture has also been shown to hold when $n=1$, $r=0$, and the generating set $S$ has size $2$, by Bryant, Herke, Maenhaut and Webb \cite{BHMW17}. In a paper in preparation \cite{EL}, the first two authors consider the general case when $n=1$ and the underlying Cayley graph is $4$-regular. Since the Cayley graph is $2$-ended, it can happen for parity reasons that no Hamilton decomposition exists. However,  this is the only obstruction, and in all other cases the Cayley graphs have a Hamilton decomposition. Together with the result of Bermond, Favaron and Maheo \cite{BFM89} for finite abelian groups, and the case $\Gamma \cong (\Z^2,+)$ of Theorem \ref{t:ZN}, this fully characterises the $4$-regular connected Cayley graphs of finite abelian groups which have Hamilton decompositions. A natural next step would be to consider the case of $6$-regular Cayley graphs. 

\begin{prob}
Let $\Gamma$ be a finitely generated abelian group and let $S$ be a generating set of $\Gamma$ such that $C(\Gamma,S)$ is $6$-regular. Characterise the pairs $(\Gamma,S)$ such that $G(\Gamma,S)$ has a decomposition into spanning double-rays.
\end{prob}

\bibliographystyle{plain}
\bibliography{Hamilton}

\end{document}